\documentclass{amsart}
\usepackage{amssymb}
\usepackage{amsmath,amsfonts,amsthm,amstext}
\usepackage[all]{xy}
\newtheorem{theorem}{Theorem}[section]
\newtheorem{lemma}[theorem]{Lemma}
\newtheorem{prop}[theorem]{Proposition}

\numberwithin{equation}{section}

\newcommand{\F}{\mathbb{F}}

\newcommand{\Z}{\mathbb{Z}}

\newcommand{\Q}{\mathbb{Q}}

\newcommand{\FP}{{\rm{FP}}}
\newcommand\ann{{\rm{ann}}}
\newcommand\Ass{{\rm{Ass}}}
\newcommand\Tor{{\rm{Tor}}}
\newcommand\rk{{\rm{rk}}\, }

\newcommand{\into}{\rightarrowtail} 
\newcommand{\onto}{\twoheadrightarrow} 
\def\ns{\vartriangleleft} 

\newtheorem{thmA}{Theorem} 
 
\newtheorem{propA}[thmA]{Proposition} 
\newtheorem{corA}[thmA]{Corollary}

\begin{document}

\title{The torsion-free rank of homology in towers of soluble pro-$p$ groups}

\author{Martin R. Bridson}

\address{Mathematical Institute,
Andrews Wiles Building,
University of Oxford,
Oxford OX2 6GG, 
UK} 
\email{bridson@maths.ox.ax.uk}

\author{Dessislava H. Kochloukova}

\address
{Department of Mathematics, State University of Campinas (UNICAMP), 
13083-859, Campinas, SP, Brazil} 
\email{desi@ime.unicamp.br}


\date{Oxford, 11 April 2016}


\begin{abstract} We show that for every finitely presented pro-$p$ nilpotent-by-abelian-by-finite group $G$  there is an upper bound on $\dim_{\mathbb{Q}_p} (H_1(M, \mathbb{Z}_p) \otimes_{\mathbb{Z}_p} \mathbb{Q}_p )$, as $M$ runs through all pro-$p$ subgroups of finite index in $G$.
\end{abstract}

\maketitle

\section{Introduction}In \cite{BK} we examined how betti numbers behave in towers of finitely presented soluble groups and proved that if
$G$ is  finitely presented  and nilpotent-by-abelian-by-finite, then there is an upper bound on $\dim_{\mathbb{Q}} H_1(M, \mathbb{Q})$, as $M$ runs through all subgroups of finite index in $G$. In this paper we shall prove the corresponding result for  pro-$p$ groups.

 By definition, a
  pro-$p$ group $G$ is finitely presented if $G \cong F / R$, where $F$ is a finitely generated free pro-$p$ group  and $R$ is a pro-$p$ subgroup of $F$ generated 
(as a topological group) by the union of finitely many conjugacy classes in $F$. 
For discrete groups, one cannot detect finite presentability using homological conditions, but
for  pro-$p$ groups one can: $G$ is  finitely presented if and only if the homology groups $H_1(G,\Z_p)$ and  $H_2(G,\Z_p)$ are finitely generated. 
Moreover,  the number of elements required
to generate $G$ as a topological group is $d(G):=\dim_{\F_p} H_1(G, \F_p)$.

In \cite{JSWilson} John Wilson proved that the Golod-Shafarevich inequality holds for soluble pro-$p$ groups. Using this,
 he proved \cite[Corollary A,(i)]{JSWilson} that for every
finitely presented soluble pro-$p$ group $G$ there is a
constant $k > 0$ such that for any pro-$p$ subgroup $U$  of finite index 
in $G$, the inequality  $d(U) \leq k [G : U]^{1/2}$ holds. He also proved that
every normal  pro-$p$ subgroup with quotient 
$\mathbb{Z}_p$  is finitely generated \cite[Corollary A, (ii)]{JSWilson}. Note how this contrasts with the discrete case:
there are many finitely presented solvable groups $N\rtimes\Z$ with $N$ not finitely generated. (In the
positive direction, 
Bieri, Neumann and Strebel \cite[Thm.~D]{BNS} proved that if  $G$ is finitely presented and soluble group with 
$\rk(G/G')\ge 2$, then
there {\em{exists}} a normal finitely generated  subgroup $N$ such that  $G / N\cong\mathbb{Z}$.)

Our first result   generalizes \cite[Cor.~A,(ii)]{JSWilson}.

\begin{propA} Let  $G$ be a finitely presented soluble pro-$p$ group. Then 
$$
\sup_{G/ H \cong \mathbb{Z}_p} d(H) < \infty.
$$
\end{propA}

The following result forms the technical heart of this paper. It is proved using methods from commutative algebra. It
will enable us
to retain the finiteness in homology given by Proposition A as we extend scalars. We remind the
reader that for a pro-$p$ group $G$ which is the inverse limit of its finite $p$-group quotients $G/U$,
homology  is defined in terms of modules over
the completed group algebra $\mathbb{Z}_p[[G]]$, which is the inverse limit  of the rings $\mathbb{Z}_p [G/U]$. Note
that $\mathbb{Z}_p[[G]]$ is a local ring whose unique maximal ideal is the kernel of the natural morphism
$\mathbb{Z}_p[[G]]\twoheadrightarrow\F_p$.

The {\em torsion-free rank} of an abelian pro-$p$ group $B$ is
$$
\rk B \ := \dim_{\mathbb{Q}_p} B {\otimes}_{\mathbb{Z}_p}\Q_p.
$$

\begin{propA}
Let $Q$ be a free abelian pro-$p$ group of finite rank and let   $A$
be a finitely generated pro-$p$ $\mathbb{Z}_p[[Q]]$-module.  
Let ${\mathcal B}$ be the set of pro-$p$ subgroups
$H \ns Q$ with $Q / H \cong \mathbb{Z}_p$. If
$
\sup_{H \in {\mathcal B}} \dim_{\mathbb{F}_p} (A {\otimes}_{\mathbb{Z}_p[[H]]} \mathbb{F}_p) < \infty, $ then
$$
\sup_{s\geq 1} \rk( A {\otimes}_{\mathbb{Z}_p
  [[Q^{p^s}]]} \mathbb{Z}_p ) < \infty.$$
 \end{propA}
	 
To prove the following theorem, we will reduce to the
metabelian case and then apply Proposition B.

\begin{thmA} Let $G$ be a finitely generated pro-$p$ group that is nilpotent-by-(torsion-free abelian).
If $
\sup_{G/ H \cong \mathbb{Z}_p} d(H) < \infty
$, then 
$$\sup_{M \in {\mathcal A}} \rk H_1(M, \mathbb{Z}_p) < \infty,$$ where ${\mathcal A}$ is the set of all pro-$p$ subgroups of finite index in $G$. 
\end{thmA}

By combining Theorem C with Proposition A and using the fact that a subgroup of finite index in a finitely presented group is finitely presented, we obtain the result stated in the abstract.

\begin{corA} For every finitely presented pro-$p$ nilpotent-by-abelian-by-finite group $G$,  there is an upper bound on $\rk H_1(M, \mathbb{Z}_p)$, as $M$ runs through all pro-$p$ subgroups of finite index in $G$.
\end{corA}

We separated Theorem C from Corollary D because, as we shall explain in Section \ref{s:examples},
there are interesting examples of metabelian pro-$p$ groups, arising in the work of
Jeremy King \cite{King2}, that are not finitely presented 
but still satisfy the conclusion of Proposition A. 
In Section \ref{s:examples} we shall also describe examples 
which show that Theorem C can fail when one changes the field
of coefficients from $\Q_p$ to $\mathbb{F}_p$.

The analogue of Corollary D in the setting of discrete groups is this: if
$G$ is  finitely presented  and nilpotent-by-abelian-by-finite, then there is an upper bound on $\dim_{\mathbb{Q}} H_1(M, \mathbb{Q})$, where $M$ runs through all subgroups of finite index in $G$. This was our main result in  \cite{BK}.
It was proved independently by Andrei  Jaikin Zapirain \cite{Jaikin} in the metabelian case under the weaker hypothesis
$\dim_{\mathbb Q} H_0(Q, H_2(N, {\mathbb Q})) < \infty$.

\medskip
{\bf Acknowledgements:} We thank Andrei Jaikin Zapirain for a suggestion that allowed us to remove
an unnecessary hypothesis from an earlier version of Theorem C and 
we thank the referee for the helpful comments that improved the paper.
 The work of the first author was supported by grants from the EPSRC
and by a Wolfson Merit Award from the Royal Society; the work of  the 
second author was supported by ``bolsa de produtividade em pesquisa", CNPq, Brazil: we thank all of these organizations.

\section{Preliminaries}

\subsection{Generating sets}

For a pro-$p$ group $G$ with a subset $X$ denote by $\langle X \rangle$ the pro-$p$ subgroup generated by $X$ i.e. the smallest pro-$p$ subgroup of $G$ that contains $X$. If 
$G = \langle X \rangle $ we say that $X$ is a generating set for $G$.

\subsection{Completed tensor products}

Let $R$ be a profinite ring, $V_1$ a profinite right $R$-module and $V_2$ a profinite left $R$-module. The profinite tensor product $V_1 \widehat{\otimes}_R V_2$ can be defined either 
by a universal property or as an inverse limit of  finite abstract tensor products \cite{RZbook}. In the case when
at least one of $V_1$ and $V_2$ is finitely presented as a profinite $R$-module, there is a natural isomorphism $V_1 \widehat{\otimes}_R V_2 \cong V_1 \otimes_R V_2$ (see \cite[Cor.~2.4]{student}). Thus, whenever possible (e.g. 
Proposition B and Theorem C)
 we use the notation $\otimes$, reserving $\widehat{\otimes}$ to emphasize the need
 for completed tensor products when we do not know whether either
 of the profinite modules is finitely presented. In this paper, 
 the only profinite modules we consider are pro-$p$ modules over pro-$p$ rings.

\subsection{Noetherian completed group algebras} \label{noetheriansubsection}

Let $Q$ be a finitely generated pro-$p$ abelian group.  $Q \cong Q_0  \oplus K$, where $K$ is a finite abelian $p$-group and $Q_0\cong \mathbb{Z}_p^n$ is free-abelian pro-$p$.
 Note that 
 $\mathbb{Z}_p[[Q_0]] \cong \mathbb{Z}_p[[t_1, \ldots , t_n]]$ is  the power series ring in $n$ commuting variables, where under the identification,  $Q_0$ is generated by $\{1 + t_1, \ldots , 1 + t_n \}$. In
 particular, $\mathbb{Z}_p[[Q_0]]$ is abstractly Noetherian. Moroever,
 since $\mathbb{Z}_p[[Q]]$ is finitely generated as an abstract $\mathbb{Z}_p[[Q_0]]$-module, 
 it too is abstractly Noetherian.

According to \cite[Lemma~7.2.2]{Wbook}, for every pro-$p$ group $G$, every pro-$p$ $\mathbb{Z}_p[[G]]$-module $V$ that is topologically generated by a finite set $X$ is generated as an abstract 
 $\mathbb{Z}_p[[G]]$-module by the same set $X$. In this sense,
 the topology of $V$ is hidden in the topology of the ring $\mathbb{Z}_p[[G]]$ that acts on $V$.

Thus we see that for a finitely generated abelian pro-$p$ group $Q$, the ring
$\mathbb{Z}_p[[Q]]$ is Noetherian in both the abstract and pro-$p$ sense. In fact,
 every abstract ideal (topology ignored)  of $\mathbb{Z}_p[[Q]]$ is finitely generated (as an abstract $\mathbb{Z}_p[[Q]]$-submodule), moreover it is closed (\cite[Lemma 7.2.2]{Wbook}), and therefore is a finitely generated pro-$p$ $\mathbb{Z}_p[[Q]]$-submodule of $\mathbb{Z}_p[[Q]]$. 
Thus when we work with ideals in $\mathbb{Z}_p[[Q]]$, we do not need to specify  that they are closed, because this condition is automatic.
This freedom to ignore the topology allows us to
 apply the methods of commutative algebra for Noetherian rings, exploiting in particular
 the fact that above every proper ideal of $\mathbb{Z}_p[[Q]]$ there are only finitely many minimal prime ideals.

\subsection{Primary decomposition in Noetherian rings}

We recall some basic ideas from commutative algebra; see \cite[Chapter~3]{eisenbud} or \cite{algebrabook}. 
Let $R$ be a commutative Noetherian ring with unit and $M$ a finitely generated $R$-module. An 
{\em{associate prime}} of $M$ is a prime ideal $P$ of $R$ such that $R/P$ arises as an $R$-submodule of $M$
(i.e. $P=\ann_R(m)$ for some $m\in M$).  The set of associated primes of $M$ is finite and is denote by $\Ass(M)$. All  of the minimal primes among the prime ideals containing $\ann_R(M)$ are associated primes of $M$.
An $R$-module $W$ is call {\em $P$-primary} if $\Ass(W) = \{ P \}$. In this case $P$ acts nilpotently on $W$.

For any $R$-submodule $V$ of $M$ there is a primary decomposition $V = V_1 \cap V_2 \cap \ldots \cap V_s$, where for each $V_i$ we have that $M / V_i$ is $P_i$-primary and the prime ideals $P_i$ are pairwise distinct 
(although one can be included in another). The prime ideals $P_1, \ldots, P_s$ are the associated primes of $M/V$.
 The component $V_i$ is unique if $P_i$ is minimal (with respect to inclusion)
 among the associated primes $\Ass(M/V)$.

We will apply the primary decomposition for $R = \mathbb{Z}_p[[Q]]$, $Q = \mathbb{Z}_p^n$ and $M$ a finitely generated pro-$p$ $R$-module. As explained in the previous subsection, we are free to ignore the topological structure of $M$ and view it as an abstract $R$-module. All abstract $R$-submodules of $M$ are finitely generated as abstract $R$-submodules and are automatically pro-$p$ $R$-submodules.

\section{Proof of Proposition A}
We have $G = H \rtimes H_0$, where $H_0 = \langle q \rangle \cong
\mathbb{Z}_p$. As we noted in the introduction, Wilson \cite[Corollary A, (ii)]{JSWilson} proved that $H$ is
finitely generated. 
On other hand, using that $H/ H' H^p$ is a finite $\mathbb{F}_p[[H_0]]$-module and $\mathbb{F}_p[[H_0]]$ is a principal ideal domain, we get a decomposition
$$
H/ H' H^p \cong (\mathbb{F}_p[[H_0]]/I_1) \oplus \ldots
\oplus  (\mathbb{F}_p[[H_0]]/I_s),
$$
where $H'$ is the pro-$p$ subgroup of $H$ generated by commutators, $H^p$ is the pro-$p$ subgroup of $H$ generated by all the set $\{ h^p \mid h \in H \}$, every $\mathbb{F}_p[[H_0]]/ I_j$ is finite, and the isomorphism is of $\mathbb{F}_p[[H_0]]$-modules. There is an isomorphism of modules 
$$
 \mathbb{F}_p[[H_0]]/I_j \cong \mathbb{F}_p[[t]]/
 (t^{i_j}),
$$
with $q$ acting as multiplication by $t+1$. (Here we have used the fact that
every ideal of $\mathbb{F}_p[[H_0]]$ is of the form $(t^i)$.) 
Let
$$
m = \max \{i_1, \ldots, i_s \} \geq 1.
$$
By \cite[Cor.~B]{JSWilson}, there is a constant $k$
such that for every pro-$p$ subgroup $U$ of finite index in $G$  we have
\begin{equation} \label{wilson11} d(U) \leq k [G : U]^{1/2}.
\end{equation}
We apply this to $$U = \langle H, q^{p^j} \rangle,
\hbox { where } 
p^j > m \geq p^{j-1} \geq 1,
$$
noting that since $U$ is a split extension  of $H$ by the pro-$p$ subgroup generated by $q^{p^j}$,
\begin{equation} \label{sexta}
U / U' U^p \cong (H / H' H^p \otimes_{\mathbb{F}_p[[\langle
  q^{p^j}  \rangle]]} \mathbb{F}_p) \oplus \mathbb{F}_p,
\end{equation}
where the second summand $\mathbb{F}_p$ is generated by the image of $q^{p^j}$.
The action of $q^{p^j}$ on $\mathbb{F}_p[[t]]/
 (t^{i_j})$ is given by multiplication with the image of  $(1 + t)^{p^j}$ which, since we are in characteristic $p$,
 is $ 1 + t^{p^j}$. Since $p^j \geq m \geq i_j$ we see that $(1 +
 t)^{p^j}$ acts as the identity. Thus
$$(H / H' H^p) \otimes_{\mathbb{F}_p[[\langle
  q^{p^j} \rangle]]} \mathbb{F}_p \cong H/ H' H^p,
$$
and by (\ref{sexta}) we have
$$
d(U) = d((H / H' H^p) \otimes_{\mathbb{F}_p[[\langle
  q^{p^j} \rangle]]} \mathbb{F}_p) + 1 = \dim (H / H' H^p) + 1 = i_1 + \ldots +
i_s + 1. 
$$
Combining this with (\ref{wilson11}) we get
$$
i_1 + \ldots +
i_s + 1  \leq k p^{j/2},
$$
thus
$$
(i_1 + \ldots +
i_s)^2 < k^2 p^j = k^2 p p^{j-1} \leq k^2 p m \leq k^2 p (i_1 + \ldots +
i_s),
$$
and 
$$
i_1 + \ldots +
i_s < k^2 p.
$$
Therefore,
$$d(H) = i_1 + \ldots +
i_s  < k^2 p.
$$
\qed

\section{Proof of Proposition B}

Throughout this section we write $R$ for the completed group algebra $\mathbb{Z}_p[[Q]]$. This is a
local ring whose unique maximal ideal is the kernel $P$ of the natural morphism
$\mathbb{Z}_p[[Q]]\twoheadrightarrow\F_p$. 

We are given a free abelian pro-$p$ group of finite rank $Q$ and 
a finitely generated $R$-module  $A$, which we know to be Noetherian. We assume
that as $H<Q$ ranges over the normal pro-$p$ subgroups with $Q/H \cong \mathbb{Z}_p$,  $\dim_{\mathbb{F}_p} (A {\otimes}_{\mathbb{Z}_p[[H]]} \mathbb{F}_p)$ remains bounded, and from this we must deduce that $ \rk( A {\otimes}_{\mathbb{Z}_p
  [[Q^{p^s}]]} \mathbb{Z}_p )$ remains bounded as the positive integer $s$ increases. Roughly speaking,
  the idea of the proof is to reduce to the case of a cyclic module (with prime annihilator),
  and we do this by examining the primary decomposition. 
  The key gain in finiteness comes in Lemma \ref{claim2}, from Noetherian
  properties of $R$ and $A$, and from Lazard's bound \cite{lazard} on the order of finite subgroups in
   ${\rm{GL}}_d(\Z_p)$. 
  
  \smallskip
  Consider the $R$-submodule  $A_m = A \Omega_m$ in $A$, where $m$ is a power of $p$
  and $\Omega_m$ is the augmentation ideal $\ker (\mathbb{Z}_p[[ Q^{m}]]\to\Z_p)$. Let
$$
A_m = N_{m,1} \cap N_{m,2} \cap \ldots \cap N_{m,j}
$$
be a primary decomposition of $A_m$ in $A$. Then
$$\Ass(A/ N_{m,i}) = \{ P_{m,i} \},
$$
where $P_{m,i}$ is a prime ideal in $R$,
and 
$$
\Ass(A / A_m) = \{ P_{m,1}, \ldots , P_{m,j} \}.
$$
Note that, by the definition of an associate prime ideal,  $R / P_{m,i}$ embeds in $A / A_m$, hence $P_{m,i}$ contains $J_m = \ann_R(A / A_m)$. 
Thus
\begin{equation} \label{inclusion1}
Q^{m} - 1 \subset J_m = \ann_R(A/ A_m) \subseteq P_{m,i}.
\end{equation} 
Now, $\mathbb{Z}_p[[Q/ Q^{m}]] = \mathbb{Z}_p[Q/ Q^{m}]$ has Krull dimension 1, so either $P_{m,i}$ is the unique maximal ideal $P \ns R$ or else $P_{m,i}$ is a minimal among the prime ideals containing $J_m$ and there is no prime ideal $P_0$  lying strictly between $P_{m,i}$ and $P$. Furthermore, by \cite[Thm.~3.1]{eisenbud},  every ideal in $R$ 
that is minimal among the prime ideals containing  $J_{m}$ is an associated prime ideal  of $A / A_m$.

The following lemma is used to limit the possibilities for the domain $R/P_{m,i}$.
Recall that ${\mathcal B}$ is the set of pro-$p$ subgroups
$H \ns Q$ with $Q / H \cong \mathbb{Z}_p$.  

\begin{lemma}\label{claim1} For each  $P_{m,i}$ there exists $H_0 \in {\mathcal B}$ such that 
$H_0 - 1 \subseteq P_{m,i}.$
\end{lemma}

\begin{proof}
 Let $F_{m,i}$ be the abstract field of fractions of the (abstract)
domain $R/ P_{m,i}$. The lemma
 is equivalent to the assertion that the image of $Q$ in $F_{m,i}$ is a cyclic group. Note that we already know that the image of $Q$ in $F_{m,i}$ is finite.

We consider two cases. First suppose that $F_{m,i}$ has
characteristic 0. 
Then  the
subfield in $F_{m,i}$ generated by $\mathbb{Q}$ and the image of $Q$
is inside the spliting field $K_{m,i}$ (over $\mathbb{Q}$) of the polynomial
$x^s - 1$, where $s$ is the exponent of the image of $Q$ in
$F_{m,i}$. By the uniqueness of the splitting field (up to isomophism),
$K_{m,i}$ is isomorphic to the subfield of $\mathbb{C}$
generated by $\mathbb{Q}$ and a primitive $s$-th root of
unity. In particular the image of $Q$ in
$F_{m,i}$ is cyclic, as claimed.

Suppose now  that $F_{m,i}$ has positive characteristic. Since the elements of $\mathbb{Z}_p \smallsetminus p \mathbb{Z}_p$ are invertible in $\mathbb{Z}_p$,
the characteristic of $F_{m,i}$ is $p$.
By (\ref{inclusion1}) $F_{m,i}$  is a finite field, hence $F_{m,i}^*$ is a cyclic group of order coprime to $p$.
It follows that the image of $Q$ in $F_{m,i}$ is trivial and $F_{m,i} = \mathbb{F}_p$, so $P_{m,i} = P$ and the lemma holds.
\end{proof}

\begin{lemma}\label{claim2}
 Let $X_m = \{ P_{m,1}, \ldots , P_{m,j} \}$  and define $X = \cup_{m \geq 1} X_m$. Then $X$ is finite.
 \end{lemma}

\begin{proof}
Consider $P_{m,i} \not= P$.
Write $W_{m,i} = A / N_{m,i}$ and note that by \cite[Prop.~3.9]{eisenbud} $P_{m,i}$ acts nilpotently on $W_{m,i}$ i.e. there is some $s_{m,i} \in \mathbb{N}$ such that $W_{m,i} P_{m,i}^{s_{m,i}} = 0$. We claim that
$$
\ann_R(V_{m,i}) = P_{m,i} \hbox{ where } V_{m,i}  = W_{m,i} / W_{m,i} P_{m,i}.
$$
To see this, note
 that for any ideal $S\ns R$ that strictly contains $P_{m,i}$, the radical
  $\sqrt{S}$ is the intersection of the minimal elements in the set of
   prime ideals containing $S$. We are in Krull dimension one, so  $\sqrt{S} = P$. 
     Since $R$ is Noetherian, 
   we have $\sqrt{S}^z \subseteq S$ for some $z \in \mathbb{N}$, and therefore $S$ has finite index in $R$ (as an abelian subgroup).
 Thus, if $\ann_R(V_{m,i})$ were to contain $P_{m,i}$ as a proper ideal, then $\ann_R(V_{m,i})$ has finite index in $R$, forcing $V_{m,i}$ and $W_{m,i}$ to be finite. But this would mean that the domain $R/P_{m,i}$
 was finite, hence a field. Since $P$ is the unique maximal ideal in $R$, this would
 imply $P_{m,i} = P$, contrary to hypothesis.

By Lemma \ref{claim1}, there is $H_0 \in {\mathcal B}$ such that $H_0 - 1 \subseteq P_{m,i}$. Then $H_0$ acts trivially on $V_{m,i}$ and so $V_{m,i} \cong V_{m,i} {\otimes}_{\mathbb{Z}_p[[H_0]]} \mathbb{Z}_p$, which is a quotient of $A {\otimes}_{\mathbb{Z}_p[[H_0]]} \mathbb{Z}_p$. 
 By assumption
$$
\sup_{H \in {\mathcal B}} \dim_{\mathbb{F}_p} (A {\otimes}_{\mathbb{Z}_p[[H]]} \mathbb{F}_p) < \infty.
$$
Recall that the {\bf minimal number of
generators} of a  pro-$p$ group $B$ is   $$d(B)= \dim_{\mathbb{F}_p} B/B' B^p.$$
 Since $P_{m,i}$ is an associated prime for $V_{m,i}$ we deduce that $R/ P_{m,i}$ embeds in $V_{m,i}$, so 
$$
d(R/P_{m,i}) \leq d(V_{m,i}) =  d(V_{m,i} {\otimes}_{\mathbb{Z}_p[[H_0]]} \mathbb{Z}_p) \leq  d(A {\otimes}_{\mathbb{Z}_p[[H_0]]} \mathbb{Z}_p) = \dim_{\mathbb{F}_p} (A{\otimes}_{\mathbb{Z}_p[[H_0]]} \mathbb{F}_p).
$$
At this stage we have established the following weak form of the conclusion $|X|<\infty$ that we seek:
$$
\delta : = \sup_{P_{m,i} \in X} d(R/P_{m,i}) \leq \sup_{H
  \in {\mathcal B}} \dim_{\mathbb{F}_p} (A
{\otimes}_{\mathbb{Z}_p[[H]]} \mathbb{F}_p) < \infty.
$$

As before, we write $F_{m,i}$ for the  abstract field of fractions of the
domain $R/P_{m,i}$. (Remember that $R= \mathbb{Z}_p[[Q]]$.)
 Let $Q_{m,i}$ be the image of $Q$ in
$R/P_{m,i}$. From Lemma \ref{claim1} we know 
that $Q_{m,i}$ is  finite cyclic, of $p$-power order.

If the characteristic of $F_{m,i}$ was positive, it would be $p$. But then, as in the proof of 
Lemma \ref{claim1}, we would get  $P_{m,i} = P$, a contradiction.
Therefore, $F_{m,i}$ has
characteristic 0, in which case 
$R/P_{m,i} \cong \mathbb{Z}_p^{d(j)}$ for some
$d(j) \leq \delta$, and $Q_{m,i}$ occurs as a finite cyclic subgroup of
${\rm{GL}}_\delta(\mathbb{Z}_p)$. This is a
virtually pro-$p$ analytic, and therefore it has  finite virtual cohomological dimension \cite{lazard}. 
(We refer to \cite{analytic} for background on pro-$p$ analytic groups.)  
It follows that there is an upper bound on the finite
subgroups of ${\rm{GL}}_d(\mathbb{Z}_p)$, yielding an integer
$r$ (a power of $p$) such that the order of each $Q_{m,i}$ always divides $r$. 

We have proved 
$$
(Q^r-1) \subseteq P_{m,i}\ \ \ \text{for all}\ \ \ \ P_{m,i}\in X,
$$
and the lemma now follows from the observation that  
in $R=\mathbb{Z}_p[[Q]]$ there are only finitely many  prime ideals 
above $(Q^{r} - 1)$; equivalently, there are only
finitely many minimal prime ideals in the local ring
$\mathbb{Z}_p[[Q]]/ (Q^{r} - 1) \cong \mathbb{Z}_p[Q/
Q^{r}]$.  
 To see that this is the case, note that since
$\mathbb{Z}_p[Q/ Q^r]$ is finitely generated as a $\mathbb{Z}_p$-module, its  Krull dimension equals that of $\mathbb{Z}_p$, which is 1. 
Thus  any prime ideal in $\mathbb{Z}_p[Q/ Q^r]$ that is not minimal is the unique maximal prime  ideal. 
And there are only finitely many minimal prime ideals.
\end{proof}

\begin{lemma}\label{claim3} For each $p$-power $m$,
define $J_m = \ann_R(A / A_m)$ and $
B_m = J_m \otimes_{{\mathbb Z}_p} {\mathbb{Q}_p}
$. Then, there is a $p$-power $m_0 \geq 1$  such that   $B_m = B_{m_0}$ for all $m\ge m_0$.
\end{lemma}

\begin{proof} 
Note that if  $m_1  \geq m_2$ are powers of $p$, then $A/ A_{m_2}$ is a quotient of $A/A_{m_1}$, hence $J_{m_1} \subseteq J_{m_2}$.
    We claim that
\begin{equation} \label{fortaleza} X_{m_2} \setminus \{ P \} \subseteq X_{m_1} \setminus \{ P \},
\end{equation} where $P$ is the unique maximal ideal of $\mathbb{Z}_p[[Q]]$. Indeed by the paragraph after (\ref{inclusion1}) any minimal prime ideal $P_{m_2,i}$ above $J_{m_2}$  such  that $P_{m_2,i} \not= P$ belongs to $X_{m_2}$ and any element of $X_{m_2} \setminus \{ P \}$ is of this type.  Thus there is no prime ideal of $\mathbb{Z}_p[[Q]]$ that contains $J_{m_1}$ and is strictly contained in $P_{m_2,i}$ (otherwise there would be a nested sequence of three  primes ideals above $J_{m_1}$, which would contradict the paragraph after (\ref{inclusion1}) with $m = m_1$). Thus $P_{m_2,i}$ is a minimal prime above $J_{m_1}$, so $P_{m_2,i} \in X_{m_1}$ and (\ref{fortaleza}) holds.

  By Lemma \ref{claim2}  $X = \cup_m X_m$ is a finite set of primes. This together with (\ref{fortaleza})  implies the existence of $m_0$, a power of $p$, 
such that 
\begin{equation} \label{fortaleza2}
X_{m_0} \setminus \{ P \} = X_{m} \setminus \{ P \} \hbox{ for } m \geq m_0 \hbox{ a power of }p.
\end{equation}
If $m\ge m_0$ is a power of $p$ and $X \not= \{ P \}$, then 
\begin{equation} \label{radical} \sqrt{J_m} = \cap_{I \in X_m} I
= 
 \cap_{I \in X_m \setminus \{ P \}} I = \cap_{I \in X_{m_0} \setminus \{ P \}} I =   \cap_{I \in X_{m_0} } I= \sqrt{J_{m_0}}\end{equation} 
because, as we noted earlier, that for any ideal $S\ns R$, the radical
  $\sqrt{S}$ is the intersection of the minimal elements in the set of
   prime ideals containing $S$.  Note that in the case $X = \{ P \}$ we have $X_m = X_{m_0} = X$, so
   \begin{equation} \label{radical21}\sqrt{J_m} = \cap_{I \in X_m} I
=   \cap_{I \in X_{m_0} } I= \sqrt{J_{m_0}} \hbox{ for } m \geq m_0 \hbox{ a }\hbox{power of } p.
   \end{equation}

   
   With equalities (\ref{radical}) and (\ref{radical21})  in hand,
   the proof of Claim 2 in \cite[Thm~3.1]{BK} applies verbatim after replacing
   $\mathbb{Q}$ with $\mathbb{Q}_p$ and $\mathbb{Z}$ with $\mathbb{Z}_p$. We repeat the
   argument for the convenience of the reader. 
It suffices to show that $B_m = B_{mp}$ for $m \geq m_0$. Fix $q \in Q$ and write $y$ for its image in $R \otimes_{\mathbb{Z}_p} {\mathbb{Q}}_p$.  Then $y^ m - 1 \in B_m$  and by (\ref{radical}) there is a positive integer $s$ such that $(y^m - 1)^s \in B_{mp}$.  The greatest common divisor 
of $x^{mp} - 1$ and $(x^m-1)^s$ in the polynomial ring ${\mathbb{Q}_p} [x]$ is $x^m-1$,
because $x^{mp}-1$ has no repeated roots in a field of charateristic 0. Therefore,
  $y^m-1 \in B_{mp}$, by Bezout's lemma, and hence $B_m = B_{mp}$. 
  \end{proof}
  
\noindent{\em{Proof of Proposition B.}} We fix a finite set of generators $a_1, \ldots , a_s$ for
$A$ as an $R$-module. 
Then $A / A_m$ is generated as an $R / J_m$-module by the images of $a_1, \ldots , a_s$ and we
have an epimorphism of $R/ J_m$-modules
$
(R/ J_m)^s \to A / A_m
$ inducing an epimorphism of finite dimensional $\mathbb{Q}_p$-modules
$$
(R/ J_m)^s \otimes_{\mathbb{Z}_p} \mathbb{Q}_p \to (A / A_m) \otimes_{\mathbb{Z}_p} \mathbb{Q}_p.
$$
Thus
$$
\dim_{\mathbb{Q}_p} (A / A_m) \otimes_{\mathbb{Z}_p} \mathbb{Q}_p \leq s \dim_{\mathbb{Q}_p} ((R/ J_m) \otimes_{\mathbb{Z}_p} \mathbb{Q}_p ).
$$
By Lemma \ref{claim3}, for $p$-powers
$m \geq m_0$  the canonical projection $R / J_{mp} \to R/ J_m$ induces an isomorphism 
$
(R / J_{mp}) \otimes_{\mathbb{Z}_p} \mathbb{Q}_p \to (R / J_m) \otimes_{\mathbb{Z}_p} \mathbb{Q}_p,
$
so 
$$
\dim_{\mathbb{Q}_p} ((R/ J_m) \otimes_{\mathbb{Z}_p} \mathbb{Q}_p) =
\dim_{\mathbb{Q}_p} ((R/ J_{m_0}) \otimes_{\mathbb{Z}_p} \mathbb{Q}_p). 
$$
Thus for ${\mathcal C} = \{ p^i \mid i \geq 1 \}$ we have
$$
\sup_{m \in {\mathcal C} }\dim_{\mathbb{Q}_p} ((A / A_m) \otimes_{\mathbb{Z}_p} \mathbb{Q}_p) \leq s. \dim_{\mathbb{Q}_p} ((R/ J_{m_0}) \otimes_{\mathbb{Z}_p} \mathbb{Q}_p) < \infty.
$$
And, by definition, $A / A_m = A\otimes_{\Z_p[[Q^m]]}\Z_p$.
\qed

\section{Proof of Theorem C}

We shall see that in the metabelian case,  Theorem C can be deduced from Proposition B by means
of straightforward calculations in homological algebra. The following proposition reduces the general
case to the metabelian case. 
For a pro-$p$ group $H$ let 
 $\{ \gamma_i(H) \}_{i \geq 1}$ be the lower central
series of $H$. By definition $\gamma_1(H) = H$,  $\gamma_{i+1}(H)$ is the closed subgroup of $H$ generated by $[\gamma_i(H), H]$.


\begin{prop} \label{p:nilpotentpro-p} Let $N \into G \onto Q$ 
be a short exact sequence of pro-$p$ groups, where $N$ is nilpotent, 
$Q$ is abelian and $G$ is finitely generated.  
Let $G_n$ be a pro-$p$
subgroup of finite index in $G$ and let $\underline{G}_n$ be the
image of $G_n$ in the metabelian group $G / N'$. Then
$$
H_1(G_n, \mathbb{Z}_p) {\otimes}_{\mathbb{Z}_p}
 \mathbb{Q}_p \cong H_1(\underline{G}_n, \mathbb{Z}_p) {\otimes}_{\mathbb{Z}_p} \mathbb{Q}_p.
$$
\end{prop}


\begin{proof}
We write $Q_n=G_n N / N$ for the image of $G_n$ in $G / N$. Set $V= N/ N'$ and $V_n = (G_n \cap N) N' / N'$.
Let  $V_{n,i}$ be the image of $G_n \cap \gamma_i(N)$ in
$\gamma_i(N) / \gamma_{i+1} (N)$. Note that $V_{n,i} \cong G_n
\cap \gamma_i(N) / G_n \cap \gamma_{i+1}(N)$ and the filtration
$\{ G_n \cap \gamma_i(N) \}_{i \geq 1}$ of $G_n \cap N$ has
quotients $V_{n,i}$.

Let $W_{n,i}$ denote the image of $G_{n} \cap \gamma_i(N)$ in
$H_1(G_n, \mathbb{Z}_p)$. 
Then, for $i \geq 2$, we have that $W_{n,i}/ W_{n, i+1}$ is an image of
$V_{n,i} / \underline{\gamma_i(N \cap G_n)}$, where
$\underline{\gamma_i(N \cap G_n)}$ is the image of
$\gamma_i(N \cap G_n)$ in $\gamma_i(N) / \gamma_{i+1} (N)$.

 In this paragraph all completed tensor products are over
 $\mathbb{Z}_p$. We think of $\gamma_i(N) / \gamma_{i+1}(N)$ as a quotient of
$\widehat{\otimes}^i V$ via the map that sends $g_1N' \widehat{\otimes} \ldots
\ldots \widehat{\otimes} g_i N'$ to $[g_1, \ldots , g_i] \gamma_{i+1} (N)$, where the commutator is left-normed. Thus we get
 $$(\gamma_i(N) / \gamma_{i+1}(N)) /  \underline{\gamma_i(N \cap G_n)}$$ as an image of $\widehat{\otimes}^i V / 
\underline{\widehat{\otimes}^i V_n}$, where $\underline{\widehat{\otimes}^i V_n}$
is the image of $\widehat{\otimes}^i V_n$ in $\widehat{\otimes}^i V$. Since $V_n$ has finite index in $V$, by induction on $i$ we see that $\widehat{\otimes}^i V / \underline{\widehat{\otimes}^i V_n}$ has finite exponent as an abelian group, hence 
 $(\gamma_i(N) / \gamma_{i+1}(N)) /  \underline{\gamma_i(N \cap G_n)}$ and its abelian subgroup  $V_{n,i} / \underline{\gamma_i(N \cap G_n)}$ have finite exponent as abelian groups for $i \geq 2$. Then
 $(V_{n,i} / \underline{\gamma_i(N \cap G_n)}) \otimes_{\mathbb{Z}_p}
 \mathbb{Q}_p = 0$, hence $(W_{n,i}/ W_{n,i+1}) \otimes_{\mathbb{Z}_p}
 \mathbb{Q}_p = 0$ for $i \geq 2$. Then $W_{n,2} \otimes_{\mathbb{Z}_p}
 \mathbb{Q}_p = 0$, so 
the image of $G_n \cap \gamma_2(N)$ in $
H_1(G_n, \mathbb{Z}_p) \otimes_{\mathbb{Z}_p}  \mathbb{Q}_p$ is trivial. 
And for the image $\underline{G}_n$ of $G_n$ in $G / N'$ we have  $H_1(G_n, \mathbb{Z}_p) \otimes_{\mathbb{Z}_p}
 \mathbb{Q}_p  \cong H_1(\underline{G}_n, \mathbb{Z}_p) \otimes_{\mathbb{Z}_p}
 \mathbb{Q}_p $.
\end{proof}


\begin{lemma} \label{novo12} Let $N \into G \onto Q$ 
be a short exact sequence of pro-$p$ groups with $G$ finitely generated.  Then  

\medskip
\noindent{\rm{(i)}}
$
d(G) \leq \dim_{\mathbb{F}_p} H_0(Q, H_1(N, \mathbb{F}_p)) + \dim_{\mathbb{F}_p} H_1(Q, \mathbb{F}_p) \leq d(G) + \dim_{\mathbb{F}_p} H_2(Q, \mathbb{F}_p);
$

\medskip
\noindent{\rm{(ii)}} $\dim_{\mathbb{Q}_p} H_1(G, \mathbb{Z}_p) \otimes \mathbb{Q}_p \leq \dim_{\mathbb{Q}_p} H_0(Q, H_1(N, \mathbb{Z}_p)) \otimes \mathbb{Q}_p + \dim_{\mathbb{Q}_p} 
H_1(Q, \mathbb{Z}_p) \otimes \mathbb{Q}_p$,

\medskip


where all tensor products are over $\Z_p$.
\end{lemma}

\begin{proof} Recall that for any finitely generated pro-$p$ group, $d(G)=\dim H_1(G,\F_p)$.
Part (i) follows immediately  from the following exact sequence (which is part of the 5-term exact sequence in homology)  $$ H_2(Q, \mathbb{F}_p) \to  H_0(Q, H_1(N, \mathbb{F}_p)) \to H_1(G, \mathbb{F}_p) \to H_1(Q, \mathbb{F}_p) \to 0.$$ Part (ii) follows from the exact sequence (which is also part of a 5-term exact sequence) $H_0(Q, H_1(N, \mathbb{Z}_p)) \to H_1(G, \mathbb{Z}_p) \to H_1(Q, \mathbb{Z}_p) \to 0$.
\end{proof}

\noindent{\bf{Proof of Theorem C.}}
We start  with a short exact sequence of pro-$p$ groups $M\into G\onto G/M$
 with $M$ nilpotent and $G/M$ torsion-free and abelian. Since $G/G'$ is a finitely generated abelian
pro-$p$ group, it is the direct sum  of a finite $p$-group $T$ and a free pro-$p$ abelian group $\Z_p^n$.
Let $N$ be the full pre-image of $T$ in $G$. 
Since $G/M$ is torsion-free abelian, $N \subseteq M$; in particular $N$ is nilpotent.
Thus we have a short exact
sequence of pro-$p$ groups $N\into G\onto \Z_p^n$. And by construction,
$N\subset H$ for every normal $H<G$ with $G/H\cong\Z_p$. 


Let $A=N/N'$.
If we write  $\underline{H}$ for the image of  $H<G$ in $\underline{G} = G/ N'$, then
 $d(\underline{H}) \leq d(H)$, hence  $\sup_{\underline{G} / \underline{H} \cong
    \mathbb{Z}_p} d(\underline{H}) < \infty$.     
This bound tells us that our hypothesis on $G$ is preserved if we replace $G$ by $G/N'$, and Proposition \ref{p:nilpotentpro-p} assures us that if we can prove the theorem for $\underline{G}=G/N'$, 
then the result will follow for $G$. Thus we assume henceforth that $A=N$ and $G=\underline{G}$.

 By applying
     Lemma \ref{novo12}(i) to the short exact sequence $A \into {H} \onto {H}/A \cong \mathbb{Z}_p^{n-1}$ we get 
$$ 
\sup_{ {G}/ {H} \cong \mathbb{Z}_p} \dim_{\mathbb{F}_p} (A {\otimes}_{\mathbb{Z}_p[[{H}]]} \mathbb{F}_p) 
\leq  
\sup_{{G}/ {H} \cong \mathbb{Z}_p} \huge[ d({H}) +  { {n-1}\choose {2}} -  n +  1
\huge] < \infty,
$$
since $H_0({H}, H_1(A,\F_p))=A {\otimes}_{\mathbb{Z}_p[[{H}]]} \mathbb{F}_p$ and $\dim H_1({H}/A,\F_p)=n-1$, and 
the binomial coefficient is the dimension of $H_2(H/A, \mathbb{F}_p)$, which    
 (as with discrete groups)
is the exterior square of the abelian  group $(H/A) / p (H/A) \cong\mathbb{F}_p^{n-1}$.
 
From Proposition B we deduce that
\begin{equation} \label{matricula} \sup_{i \geq 1} \rk( A {\otimes}_{\mathbb{Z}_p [[Q^{p^i}]]} \mathbb{Z}_p ) < \infty.\end{equation}
\noindent
Let $G_m$ be a pro-$p$ subgroup  of finite index in $G$ and consider the following exact sequence\footnote{The use of $\otimes$ rather than $\hat\otimes$ is permitted because the trivial
module $\Z_p$ is finitely presented over $\Z_p[[G]]$ for every finitely generated pro-$p$ group $G$.} 
$$ 
\Tor_1^{\mathbb{Z}_p[[G_m/ A \cap G_m]]}(A / A \cap G_m, \mathbb{Z}_p) \to    (A \cap G_m)  {\otimes}_{\mathbb{Z}_p[[G_m/ A \cap G_m]]} \mathbb{Z}_p \to  $$ $$ A  {\otimes}_{\mathbb{Z}_p[[G_m/ A \cap G_m]]} \mathbb{Z}_p \to (A / A \cap G_m) {\otimes}_{\mathbb{Z}_p[[G_m/ A \cap G_m]]} \mathbb{Z}_p  \to 0.$$
Because $A / A \cap G_m$ is finite, the first and last groups in this sequence are torsion, so
$$
\rk((A \cap G_m ) {\otimes}_{\mathbb{Z}_p[[G_m/ A \cap G_m]]} \mathbb{Z}_p)  = \rk(A  {\otimes}_{\mathbb{Z}_p[[G_m/ A \cap G_m]]} \mathbb{Z}_p).
$$
Then, by applying Lemma \ref{novo12}(ii) to $A \cap G_m \into G_m \onto G_m/ (A \cap G_m)$, we get
$$
\begin{aligned}
\dim_{\mathbb{Q}_p} (H_1(G_m, \mathbb{Z}_p) \otimes_{\mathbb{Z}_p} \mathbb{Q}_p ) 
&\leq n + \rk((A \cap G_m)  {\otimes}_{\mathbb{Z}_p[[G_m/ A \cap G_m]]} \mathbb{Z}_p)\\
  &= n +  \rk(A   {\otimes}_{\mathbb{Z}_p[[G_m/ A \cap G_m]]} \mathbb{Z}_p),
\end{aligned}
$$
where $n$ is the rank of $G_m/(A\cap G_m)$, which has finite index in $Q  = G / A$.
Finally, taking a $p$-power $s$ such that $Q^s \subseteq G_m / (A \cap G_m) \subseteq Q$, we have 
$$\rk (A  {\otimes}_{\mathbb{Z}_p[[G_m/ A \cap G_m]]} \mathbb{Z}_p) \leq \rk (A  {\otimes}_{\mathbb{Z}_p[[Q^s ]]} \mathbb{Z}_p)$$ and hence
$$
\dim_{\mathbb{Q}_p} (H_1(G_m, \mathbb{Z}_p) \otimes_{\mathbb{Z}_p} \mathbb{Q}_p ) \leq n + \rk (A  {\otimes}_{\mathbb{Z}_p[[Q^s ]]} \mathbb{Z}_p).$$
Then (\ref{matricula}) completes the proof.
\qed

\section{Examples}\label{s:examples}

In this section we shall present two examples to show that Theorem C fails if the coefficient field $\Q_p$ is replaced
by $\mathbb{F}_p$. We shall then explain an example due to Jeremy King \cite{King3} which shows that
Theorem C is of interest beyond the setting of finitely presented groups.

As we noted in the introduction, finite presentability for pro-$p$
groups is essentially a  homological condition. More precisely,  
a pro-$p$ group $G$ is finitely presented if and only if is of type $\FP_2$. In
the pro-$p$ setting, a group $G$ is defined to be of type $\FP_m$ if, as a trivial $\mathbb{Z}_p [[G]]$-module, $\mathbb{Z}_p$ has a projective resolution with the modules finitely generated in dimensions $\le m$.
Since $\mathbb{Z}_p[[G]]$ is a local ring, it is easy to see that $G$ is of type $\FP_m$ if and only if the pro-$p$ homology groups $H_i(G, \mathbb{F}_p)$ are finite for $i \leq m$.  

In \cite{King2} King showed that if $N$ is a normal pro-$p$ subgroup of $G$, with $G/ N$ of finite rank (in which case $\mathbb{Z}_p[[G/N]]$ is left and right Noetherian), 
then $G$ is of type $\FP_m$ if and only if $H_i(N, \mathbb{Z}_p)$ is finitely generated as 
a pro-$p$ $\mathbb{Z}_p[[G/N]]$-module for $i \leq m$, where the action of $G/N$ is induced by conjugation. In particular, if $G$ is a finitely generated metabelian pro-$p$ group, i.e.  there is a short exact sequence of pro-$p$
groups  $A \to G \to Q $ with $A$ and $Q$ abelian, then $G$ 
will be finitely presented if and only if $H_2(A, \mathbb{Z}_p) \cong A \widehat{\wedge}
A$ is finitely generated as a $\mathbb{Z}_p [[Q]]$-module,
where $Q$ acts diagonally on the completed exterior square $A \widehat{\wedge} A$.

This criterion for finitely presentability of metabelian pro-$p$ groups plays a crucial role in our examples: it is
hidden in our appeals to King's papers \cite{King1}, \cite{King3}.

\subsection{The first example: switching to $\mathbb{F}_p$ when $p$ is odd}

Suppose $p>2$
and consider the pro-$p$ group $G = A \rtimes Q$  where $A = \mathbb{F}_p[[t]]$
and  $Q = \langle x , y \rangle = \mathbb{Z}_p \times \mathbb{Z}_p$, with $x$ acting
as multiplication by $1 + t$ while $y$ acts as multiplication by $1 + 2t$.
By \cite[Thm.~A]{King3},   $G$ is finitely presented. And since $(1+\tau)^{p^i} = 1 + \tau^{p^i}$ in characteristic
$p$, we have
$$
A \otimes_{\mathbb{F}_p[[Q^{p^i}]]} \mathbb{F}_p \cong \mathbb{F}_p[[t]]/ ((1 + t)^{p^i} - 1, (1 + 2t)^{p^i} - 1) = \mathbb{F}_p[[t]]/ (t^{p^i}).
$$
 Hence
$$\sup_{M \in {\mathcal A}} \dim_{\mathbb{F}_p} (H_1(M, \mathbb{Z}_p) \otimes_{\mathbb{Z}_p} \mathbb{F}_p ) \geq  \sup_{i \geq 1}   \dim_{\mathbb{F}_p} (H_1(A \rtimes Q^{p^i}, \mathbb{Z}_p) \otimes_{\mathbb{Z}_p} \mathbb{F}_p ) = $$ $$ 2 + \sup_{i \geq 1}   \dim_{\mathbb{F}_p} (A \otimes_{\mathbb{F}_p[[Q^{p^i}]]} \mathbb{F}_p)  = 2 + \sup_{i \geq 1}  p^i  = \infty,$$ where ${\mathcal A}$ is the set of all pro-$p$ subgroups of finite index in $G$.
  
\subsection{The second example: switching to $\mathbb{F}_p$ when $p = 2$}

Consider the pro-$2$ group $G = A \rtimes Q$, where $A = \mathbb{F}_2[[t]]$
and  $Q = \langle x , y \rangle = \mathbb{Z}_2 \times \mathbb{Z}_2$, with $x$ acting
as multiplication by $1 + t$ while $y$ acts as multiplication by $1 + t + t^2$. Again,  $G$ is finitely presented by \cite[Thm.~B]{King3} and
$$
A \otimes_{\mathbb{F}_2[[Q^{2^i}]]} \mathbb{F}_2 \cong \mathbb{F}_2[[t]]/ ((1 + t)^{2^i} - 1, (1 + t + t^2)^{2^i} - 1) = \mathbb{F}_2[t]]/ (t^{2^i}).
$$
 Hence
$$\sup_{M \in {\mathcal A}} \dim_{\mathbb{F}_2} (H_1(M, \mathbb{Z}_2) \otimes_{\mathbb{Z}_2} \mathbb{F}_2 ) \geq  \sup_{i \geq 1}   \dim_{\mathbb{F}_2} (H_1(A \rtimes Q^{2^i}, \mathbb{Z}_2) \otimes_{\mathbb{Z}_2} \mathbb{F}_2 ) = $$ $$ 2 + \sup_{i \geq 1}   \dim_{\mathbb{F}_2} (A \otimes_{\mathbb{F}_2[[Q^{2^i}]]} \mathbb{F}_2)  = 2 + \sup_{i \geq 1}  2^i  = \infty,$$ where ${\mathcal A}$ is the set of all pro-$2$ subgroups of finite index in $G$.
 
\subsection{King's example}  Let $p$ be an odd prime, let
 $Q = \langle x , y \rangle = \mathbb{Z}_p \times \mathbb{Z}_p$, 
 let 
 $A = \mathbb{F}_p[[Q]]/ (x+ x^{-1} + y + y^{-1} - 4)$ and, following  
\cite{King1}, consider the  pro-$p$ group $G=A \rtimes Q$. By \cite[Prop.~3.4]{King1} $G$ is not finitely presented and by \cite[Prop.~3.5]{King1} for every normal pro-$p$ subgroup $N$ of $G$ such that $G/ N \cong\mathbb{Z}_p$ we have that $N$ is finitely generated i.e. $d(N) < \infty$. Our aim is to show that
$$
\sup_{G/N\cong\Z_p} d(N) < \infty.
$$
Theorem C can then be applied to $G$, even though $G$ is not finitely presented.

To this end, observe first that  $N = A \rtimes H$, where $H$ is a pro-$p$ subgroup of $Q$ with  $Q/ H \cong \mathbb{Z}_p$.
According to Lemma \ref{novo12}(i), we will be done if we can bound the dimension of
$$
H_0(H,\, H_1(A, {\mathbb{F}_p})) = A \widehat{\otimes}_{\mathbb{F}_p[[H]]} \mathbb{F}_p.
$$
 Since the relation that defines $A$ is symmetric with respect to $x$ and $y$, we can assume that $H = \langle x y^{- \lambda} \rangle$ for some $\lambda \in \mathbb{Z}_p$. 
 Indeed we have that $H = \langle x^{i} y^{j} \rangle$ with $i,j \in \mathbb{Z}_p$ such that the ideal of $\mathbb{Z}_p$ generated by $i$ and $j$ is the whole ring $\mathbb{Z}_p$. Then $i \notin p \mathbb{Z}_p$ or $j \notin p \mathbb{Z}_p$. Since $\mathbb{Z}_p \setminus p \mathbb{Z}_p$ is the set of invertible elements in $\mathbb{Z}_p$ we can assume that $i$ or $j$ is invertible; say $i$ is invertible. Then for $\lambda =  - j i^{-1}$
 $$
 H =  \langle x^{i} y^{j} \rangle =  \langle (x y^{j i^{-1}})^i \rangle =  \langle x y^{j i^{-1}} \rangle =  \langle x y^{ - \lambda} \rangle.
  $$
    Then
$$
A \widehat{\otimes}_{\mathbb{F}_p[[H]]} \mathbb{F}_p \cong\mathbb{F}_p[[Q]] / (x+ x^{-1} + y + y^{-1} - 4, x y ^{- \lambda} - 1 ) \cong\mathbb{F}_p[[\langle y \rangle]] / (y^{\lambda} + y^{- \lambda} + y + y^{-1} - 4)$$
where the last isomorphism sends $x$ to $y^{\lambda}$. Since $\mathbb{F}_p[[\langle y \rangle]] $ is the power series ring $\mathbb{F}_p[[t]]$ with $y = 1 + t$ we deduce that
$$
A \widehat{\otimes}_{\mathbb{F}_p[[H]]} \mathbb{F}_p \cong
\mathbb{F}_p[[t]]/ (f_\lambda)
$$ 
where $$f_\lambda = 
(1 + t)^{\lambda}  + (1+t)^{ - \lambda} + (1 + t) + (1 + t)^{-1} - 4.$$
We claim that as $\lambda$ varies, the first non-zero coefficient of $f_\lambda$ appears
in degree less than $p$, in other words,
\begin{equation}\label{key}
\sup_\lambda \inf \{i \mid f_\lambda \not\in t^{i+1}\mathbb{F}_p[[t]]\} \ < \ p.
\end{equation}
 With this bound in hand,
 writing $f \in z t^i + t^{i+1} \mathbb{F}_p[[t]]$ with $z \in \mathbb{F}_p \setminus \{ 0 \}$, 
we have $(f) = (t^i)$ and 
$$
\dim_{\mathbb{F}_p} A \widehat{\otimes}_{\mathbb{F}_p[[H]]} \mathbb{F}_p  = \dim_{\mathbb{F}_p} \mathbb{F}_p[[t]]/ (t^i) = i <p.
$$
The rest  of the proof is devoted to establishing the bound in (\ref{key}).
Given $\lambda\in \mathbb{Z}_p$, we write
$$
\lambda = z_0 + p \lambda_1 \hbox{ and } - \lambda  = a_0 + p \lambda_2
$$
where $z_0, a_0 \in \{ 0, 1, \ldots, p-1 \}$ and $ \lambda_1, \lambda_2 \in \mathbb{Z}_p$.
As $(1 + t)^{p \lambda_j} = (1 + t^p)^{\lambda_j}$ in $\mathbb{F}_p[[t^p]]$, we have
 $$
 (1 + t)^{\lambda} = (1 + t)^{z_0} ( 1 + t)^{p \lambda_1} \in (1 + t)^{z_0} + t^p \mathbb{F}_p[[t]]$$
 and
 $$
 (1 + t)^{- \lambda} = (1 + t)^{a_0} ( 1 + t)^{p \lambda_2} \in (1 + t)^{a_0} + t^p \mathbb{F}_p[[t]].$$
 Note that if   $z_0 > 0$ then $a_0+z_0=p$. In particular, taking $\lambda=z_0=1$ we have
$$
(1 + t)^{-1} \in (1 + t)^{p-1} + t^p \mathbb{F}_p[[t]].$$
In general, 
\begin{equation} \label{last} f_\lambda \in g_\lambda + t^p \mathbb{F}_p[[t]],
\end{equation}
where $g_\lambda =  (1 + t)^{z_0}+ (1 + t)^{a_0} + (1 + t) + ( 1 + t)^{p-1} - 4$.
To establish (\ref{key}) we must argue that $g_\lambda$ is a non-zero polynomial.
If $z_0 = 0$ then $a_0 = 0 $ and $g_\lambda = t + (1 + t)^{p-1} - 1  $, which is a non-zero element of $\mathbb{F}_p[t]$ because
$p>2$. If $z_0 > 0$ then 
$$
g_\lambda = (1 + t)^{z_0} + (1+ t)^{p- z_0} + (1 + t) + (1 + t)^{p-1} - 4,
$$
which is again non-zero. This completes the proof.

By calculating with binomial coefficients, one can replace (\ref{key}) by a bound that is independent of $p>2$.

 \end{document}